\documentclass[12pt,english]{amsart}

\usepackage{color}

\usepackage[paperwidth=216mm,paperheight=279mm,inner=3.5cm,outer=3.5cm,top=2.5cm,bottom=2.6cm]{geometry}

\usepackage[T1]{fontenc}
\usepackage{amssymb}

\usepackage{paralist}
\usepackage{babel}
\usepackage[leqno]{amsmath}
\usepackage{amsthm}
\usepackage{arydshln}
\usepackage[colorlinks,linkcolor=black,citecolor=black,urlcolor=black,hypertexnames=true]{hyperref}
\newcommand\mbb{\mathbb}

\newcommand\mcal{\mathcal}

\newcommand\sF{\mcal{F}}

\newcommand\sH{\mcal{H}}

\newcommand\sM{\mcal{M}}

\newcommand\sR{\mcal{R}}

\newcommand\sV{\mcal{V}}

\newcommand\C{\mbb{C}}

\renewcommand\P{\mbb{P}}

\newcommand\R{\mbb{R}}

\DeclareMathOperator\SL{SL}

\DeclareMathOperator\interior{int}

\renewcommand\epsilon{\varepsilon}
\renewcommand\ge{\geqslant}

\renewcommand\le{\leqslant}

\renewcommand\phi{\varphi}

\renewcommand\theta{\vartheta}

\DeclareMathOperator\Diag{Diag}
\DeclareMathOperator\Disc{Disc}
\DeclareMathOperator\Sym{Sym}

\newcommand\randN{N}
\newcommand\randF{F}

\theoremstyle{plain}
\newtheorem{Thm}{Theorem}
\newtheorem{Prop}[Thm]{Proposition}
\newtheorem{Cor}[Thm]{Corollary}

\newtheorem{Conjecture}[Thm]{Conjecture}

\newtheorem*{Thm*}{Theorem}
\newtheorem*{Prop*}{Proposition}
\newtheorem*{Cor*}{Corollary}
\newtheorem*{Lemma*}{Lemma}
\newtheorem*{Sublemma*}{Sublemma}
\newtheorem*{Conjecture*}{Conjecture}

\theoremstyle{definition}

\newtheorem{Example}[Thm]{Example}

\newtheorem{Remark}[Thm]{Remark}

\newtheorem*{Constr*}{Construction}
\newtheorem*{Def*}{Definition}
\newtheorem*{Defs*}{Definitions}
\newtheorem*{Example*}{Example}
\newtheorem*{Examples*}{Examples}
\newtheorem*{Exercise*}{Exercise}
\newtheorem*{LemmaDef*}{Lemma and Definition}
\newtheorem*{Notation*}{Notation}
\newtheorem*{Problem*}{Problem}
\newtheorem*{Question*}{Question}
\newtheorem*{Remark*}{Remark}
\newtheorem*{Remarks*}{Remarks}
\newtheorem*{Warning*}{Warning}

\newtheorem*{Text*}{}

\numberwithin{equation}{section}
\numberwithin{Thm}{section}

\DeclareMathOperator\Orth{O}

\begin{document}
\title[Determinantal representations via homotopy continuation]{Determinantal representations of hyperbolic curves via polynomial homotopy continuation}

\author{Anton Leykin}
\address{School of Mathematics, Georgia Institute of Technology,
  Atlanta, GA, USA}
\email{leykin@math.gatech.edu}
\thanks{AL was supported by NSF grant DMS-1151297}

\author{Daniel Plaumann}
\address{Fachbereich Mathematik und Statistik, Universit\"at Konstanz,
  Germany}
\email{daniel.plaumann@uni-konstanz.de}

\date{\today}

\begin{abstract}
  A smooth curve of degree $d$ in the real projective plane is
  hyperbolic if its ovals are maximally nested, i.e.~its real points
  contain $\lfloor\frac d2\rfloor$ nested ovals. By the
  Helton-Vinnikov Theorem, any such curve admits a definite symmetric
  determinantal representation. We use polynomial homotopy
  continuation to compute such representations numerically. Our method
  works by lifting paths from the space of hyperbolic polynomials to a
  branched cover in the space of pairs of symmetric matrices.
\end{abstract}

\maketitle

\section*{Introduction}

Let $p\in\C[t,x,y]$ be a homogeneous polynomial of degree $d\ge 1$. A
\emph{(linear symmetric) determinantal representation} of $p$ is an expression
\[
p=\det(tM_1+xM_2+yM_3)
\]
where $M_1,M_2,M_3$ are complex symmetric matrices of size $d\times
d$.  Determinantal representations of plane curves are a classical
topic of algebraic geometry. Existence for smooth curves of arbitrary
degree was first proved by Dixon in 1902 \cite{Di}. For an exposition
in modern language, see Beauville \cite{Be}.

Real determinantal representations of real curves have only been
studied systematically much later in the work of Dubrovin \cite{Du}
and Vinnikov \cite{Vi89}. Of particular interest here are the
\emph{definite representations}, where some linear combination of the
matrices $M_1,M_2,M_3$ is positive definite. By a celebrated result
due to Helton and Vinnikov \cite{HV}, these correspond exactly to the \emph{hyperbolic curves},
whose real points consist of maximally nested ovals in the real
projective plane.

The Helton-Vinnikov theorem (previously known as the Lax Conjecture)
has attracted attention in connection with semidefinite programming,
since it characterizes the boundary of those convex subsets of the real plane that can
be described by linear matrix inequalities. See Vinnikov \cite{Vi11}
for an excellent survey.

While the Helton-Vinnikov theorem ensures the existence of a definite
determinantal representation for any hyperbolic curve, finding such a
representation for a given polynomial $p$ remains a difficult
computational problem. With a suitable choice of coordinates, we can
restrict to representations of the form
\[
p=\det(tI_d+xD+yR)
\]
where $I_d$ is the identity matrix, $D$ is a real diagonal and $R$ a
real symmetric matrix. The hyperbolicity of $p$ is reflected in the
fact that for any point $(u,v)\in\R^2$, all roots of the univariate
polynomial $p(t,u,v)\in\R[t]$ are real. Given such $p$, the
computational task of finding the unknown entries of $D$ and $R$
leads, in general, to a zero-dimensional system of polynomial
equations. However, as $d$ grows, this direct approach quickly becomes
infeasible in practice. This, as well as symbolic methods and an
alternative approach via theta functions based on the proof of the
Helton-Vinnikov theorem, have been investigated in \cite{PSV}. As far
as actual computations are concerned, $d=6$ was the largest degree for
which computations terminated in reasonable time.

We present here a more sophisticated numerical approach,
implemented with \textsc{NAG4M2}: the \emph{NumericalAlgebraicGeometry}
package~\cite{Leykin:NAG4M2} for \textsc{Macaulay2}~\cite{M2www}. 
We consider the branched cover of the space of
homogeneous polynomials by pairs of matrices $(D,R)$, with $D$
diagonal and $R$ symmetric, via the determinantal map 
$$(D,R)\mapsto \det(tI_d+xD+yR).$$ 
We use known results on the number of equivalence
classes of complex determinantal representations to show that the
determinantal map is unramified over the set of smooth hyperbolic
polynomials (Thm.~\ref{Thm:U}). We then use the fact that this set is
path-connected. In fact, an explicit path connecting any hyperbolic
polynomial to a certain fixed polynomial was constructed by Nuij in
\cite{Nu}, which we refer to as the N-path. Our algorithm works by
constructing a lifting of the N-path to the covering space. The
advantage over an application of a blackbox homotopy continuation solver 
to the zero-dimensional system of equations is that we need to track a {\em single}
path instead of as many paths as there are complex solutions. We also recover the
Helton-Vinnikov theorem from these topological considerations and the
count of equivalence classes of complex representations.

Since the singular locus has codimension at least $2$ inside the set
of strictly hyperbolic polynomials, the N-path avoids singularities
for almost all starting polynomials (Prop. \ref{Prop: avoid
  L}). In the unlikely event that the N-path goes through the
singular locus, it is possible to perturb the starting point and
obtain an approximate determinantal representation. We also provide an
algorithm that produces a complex determinantal representation via the
complexification of the N-path.

We modify the approach of Nuij to introduce a {\em randomized N-path},
which depends on a choice of random linear forms. For a given starting
polynomial $p$, we conjecture (Conjecture~\ref{conjecture:main}) that the randomized N-path avoids the
singular locus with probability 1. This also results in a better
practical complexity of the computation than the original N-path
(Remark~\ref{Remark:N-path-comparison}).

Our proof-of-concept implementation is written in the top-level interpreted language of
{\sc Macaulay2} and, by default, uses standard double
floating point precision. Even with these limitations we can compute 
small examples in reasonable time (see the example in \S\ref{section:example} for $d=6$): we are able to finish examples with $d\leq 10$ within one day. With arbitrary precision
arithmetic and speeding up the numerical evaluation procedure, we see
no obstacles to computing robustly for $d$ in double digits using the
present-day hardware.

We note that the developed method constructs an {\em intrinsically real} homotopy to find real solutions to a (specially structured) polynomial system. The only other intrinsically real homotopy known to us is introduced for Khovanskii-Rolle continuation in~\cite{BS}. 

\emph{Acknowledgements.} We are grateful to Institut Mittag-Leffler,
where this project has started, for hosting us in the Spring of
2011. We would like to thank Greg Blekherman and Victor Vinnikov for helpful discussions.

\section{Hyperbolic and determinantal polynomials}
\noindent We consider real or complex homogeneous polynomials of degree $d\ge 1$ in $n+1$ variables
$(t,x)$, $x=x_1,\dots,x_n$. Let
\begin{align*}
\sF &=\bigl\{p\in\C[t,x]\:|\: p\text{ is homogeneous of total degree }d\text{
  and }p(1,0,\dots,0)=1\bigr\}\\
\sF_\R &= \sF\cap\R[t,x].
\end{align*}

A polynomial $p\in \sF_\R$ is called \emph{hyperbolic} if all roots of
the univariate polynomial $p(t,u)\in\R[t]$ are real, for all
$u\in\R^n$. It is called \emph{strictly hyperbolic} if all these roots
are distinct, for all $u\in\R^n$, $u\neq 0$.  Write
\[
\sH =\bigl\{p\in \sF_\R\:|\: p\text{ is hyperbolic}\bigr\}.\\
\]

\begin{Prop}\label{Prop:TopologyH}\mbox{}
  \begin{enumerate}
  \item The interior $\interior(\sH)$ of $\sH$ is the set of strictly hyperbolic polynomials
and $\sH$ is the closure of $\interior(\sH)$ in $\sF_\R$.
   \item The set $\interior(\sH)$ is contractible and path-connected
     (hence so is $\sH$).
 \item A polynomial $f\in\sH$ is strictly hyperbolic if and only if the
    projective variety $\sV_\C(f)$ defined by $f$ has no real singular
    points.
  \item Let $\sH^\circ$ be the set of hyperbolic polynomials $p\in\sH$ for
    which $\sV_\C(p)$ is smooth. Then $\interior(\sH)\setminus\sH^\circ$ has codimension at
    least $2$ in $\sF_\R$.
  \end{enumerate}
\end{Prop}

\begin{proof}
  (1) and (2) are proved by Nuij \cite{Nu} (see also Section
  \ref{Sec:N-path} below). (3) is proved in \cite[Lemma 2.4]{PV}. (4)
  follows from the fact that the elements of $\interior(\sH)$ have no
  real singularities, while complex singularities must come in
  conjugate pairs.
\end{proof}

For the remainder of this section, we restrict to the case $n=2$
(plane projective curves) and use $(x,y)$ instead of $(x_1,x_2)$.

\begin{Remark}\label{Remark:StrictlyHyperbolicSingular}
  If $n=2$ and $d\le 3$, then $\sH^\circ=\interior(\sH)$, i.e.~every
  strictly hyperbolic curve of degree at most $3$ is smooth. This is
  simply because a real plane curve of degree at most $3$ cannot
  have any non-real singularities. 

  When $d\ge 4$, a strictly hyperbolic curve may still have complex
  singularities. For example, let $p=1/19(19t^4 -31 x^2t^2 - 86 y^2t^2
  + 9x^4 + 41 x^2 y^2 +39 y^4)$. One can check that $p$ is hyperbolic
  and that the projective plane curve defined by $p$ has no real
  singularities, hence $p$ is strictly hyperbolic. However, $(1:\pm
  2:\pm i)$ are two pairs of complex-conjugate singular points of
  $\sV_\C(p)$. Thus $p\in\interior(\sH)\setminus\sH^\circ$.
\end{Remark}

\noindent We will use the notation
\begin{align*}
\sM&=\bigl\{(D,R)\in\bigl(\Sym_d(\C))^2\:|\: D\text{ is
  diagonal}\bigr\},\\
\sM_\R&=\sM\cap(\Sym(\R))^2.
\end{align*}
Note that, since $n=2$, we have $\dim_\C\sF=\dim_\R\sF_\R=\dim_\R\sM_\R=\dim_\C\sM=\frac{d(d+3)}{2}$. We study the map
\[
\Phi\colon\left\{
  \begin{array}{ccc}
    \sM & \to & \sF\\
    (D,R) & \mapsto & \det(tI_d+xD+yR).
  \end{array}\right.
\]
and its restriction to $\sM_\R$.

The image of $\sM_\R$ under $\Phi$ is contained in $\sH$. It is also not hard
to show that it is closed (see \cite[Lemma 3.4]{PV}).
Our first goal is to find a connected open subset $U$ of
$\sH$ such that the restriction of $\Phi$ to $\Phi^{-1}(U)$ is smooth.

For fixed $p\in\sH$, the group
$\SL_d(\C)\times\{\pm 1\}$ acts on the determinantal representations
$p=\det(tM_1+xM_2+yM_3)$ via symmetric
equivalence. In other words, any $A\in\SL_d(\C)\times\{\pm 1\}$ gives a new
representation $p=\det(tAM_1A^T+xAM_2A^T+yAM_3A^T)$. When we restrict
to the normalized representations we are considering, we have an
action on pairs $(D,R)\in\Phi^{-1}(p)$ by those elements
$A\in\SL_d(\C)\times\{\pm 1\}$
for which $AA^T=I_d$ (i.e.~$A\in\Orth_d(\C)$) and $ADA^T$ is
diagonal.

\begin{Thm}
  For $n=2$, any $p\in\sF$ has only finitely many complex representations
  $p=\det(tI_d+xD+yR)$ up to symmetric equivalence. If the curve
  $\sV_\C(p)$ is smooth, the number of equivalence classes is precisely
  $2^{g-1}\cdot(2^g+1)$, where $g=\binom{d-1}{2}$ is the genus of
  $\sV_\C(p)$.
\end{Thm}

\begin{proof}
  For smooth curves, the equivalence classes of symmetric
  determinantal representations are in canonical bijection with
  ineffective even theta characteristics; see \cite[Thm.~2.1]{PSV} and
  references given there.
\end{proof}

\begin{Thm}\label{Thm:U}
  The set $\sH^\circ$ of smooth hyperbolic polynomials in three
  variables is an open,
  dense, path-connected subset of $\sH$, and each fibre of $\Phi$ over a
  point of $\sH^\circ$ consists of exactly $2^{g-1}\cdot(2^g+1)\cdot
  2^{d-1}\cdot d!$ distinct points.
  \end{Thm}

  \begin{proof} The statements about the topology of $\sH^\circ$
    follow immediately from Prop.~\ref{Prop:TopologyH}. Let
    $p\in\interior(\sH)$ and let $(D,R)\in\Phi^{-1}(p)$, which means
    $p=\det(tI_d+xD+yR)$. The diagonal entries of $D$ are the
    zeros of $p(t,-1,0)$. Since $p$ is strictly hyperbolic, these
    zeros are real and distinct. So $D$ is a real diagonal matrix with
    distinct entries. It follows then that the centralizer of $D$ in
    $\Orth_d(\C)$ consists precisely of the $2^d$ diagonal matrices
    with entries $\pm 1$. Let $S$ be such a matrix with $S\neq \pm
    I_d$. We want to identify the set of symmetric matrices $R$ that
    commute with $S$. Up to permutation, we may assume that the first
    $k$ diagonal entries of $S$ are equal to $-1$ and the remaining
    $d-k$ are equal to $1$. It follows then that any $R$ with $SR=RS$
    must have $r_{ij}=r_{ji}=0$ if $i>k\ge j$, so that $R$ is
    block-diagonal. For such $R$ to show up in a pair
    $(D,R)\in\Phi^{-1}(p)$, the polynomial $p$ must be reducible. In
    particular, if $p\in\sH^\circ$, there is no such $S$ commuting
    with $R$. It follows then that $\{SRS\:|\: S\text{ diagonal with
    }S^2=I_d\}$ has $2^{d-1}$ distinct elements. Permuting the
    distinct diagonal entries of $D$ gives $d!$ possible choices of
    $D$. This, combined with the count of equivalence classes in the
    preceding theorem, completes the proof.
\end{proof}

\begin{Cor}
  The restriction of $\Phi$ to $\Phi^{-1}(\sH^\circ)$ is smooth.
\end{Cor}

\begin{proof}
The restriction of $\Phi$ to $\Phi^{-1}(\sH^\circ)$ is a polynomial map with
finite fibres that is unramified over $\sH^\circ$, since the cardinality of
the fibre does not change. Hence it is smooth (see for example
Hartshorne \cite[III.10]{Ha}).
\end{proof}

\noindent We sketch an argument for deducing the Helton-Vinnikov theorem from Thm.~\ref{Thm:U}.

\begin{Cor}[Helton-Vinnikov Theorem] Every hyperbolic polynomial
  $p\in\sH$ in three variables admits a determinantal representation
  $p=\det(tI_d+xD+yR)$ with $D$ diagonal and $R$ real symmetric.
\end{Cor}

\begin{proof}
  Since all fibres of $\Phi$ over $\sH^\circ$ have the same
  cardinality and $\sH^\circ$ is path-connected, the number of real
  points in each fibre must also be constant over $\sH^\circ$. That
  number cannot be zero, since there exist fibres with real
  points. (This amounts to showing that for each $d\ge 1$ there exists
  a real pair $(D,R)$ such that $p=\det(tI_d+xD+yR)$ defines a smooth
  curve. This can for example be deduced with the help of Bertini's
  theorem.) It follows that $\sH^\circ$ is contained in
  $\Phi(\sM_\R)$. On the other hand, $\Phi(\sM_\R)$ is closed in
  $\sF_\R$ by \cite[Lemma 3.4]{PV} and contained in $\sH$, hence
  $\Phi(\sM_\R)=\sH$.
\end{proof}

\begin{Remark}\label{rem: real repr}
 The number of equivalence classes of real definite representations of
 a hyperbolic curve is in fact also known, namely it is $2^g$. See
 \cite{PSV} and references to \cite{Vi93} given there. We conclude
 that $\Phi^{-1}(p)\cap\sM_\R$ consists of $2^g\cdot 2^{d-1}\cdot d!$
 distinct points for every $p\in\sH^\circ$.

  Note also that, even if $p$ is hyperbolic, it will typically admit real
  determinantal representations $p=\det(tM_1+xM_2+yM_3)$ that are not
  definite, i.e.~are not equivalent to such a representation with
  $M_1=I_d$ and $M_2,M_3$ real. Such representations do not reflect
  the hyperbolicity of $p$.
\end{Remark}

\section{The Nuij path}\label{Sec:N-path}
In order to use homotopy continuation methods for numerical
computations, we need an explicit path connecting any two given points
in the space $\sH$ of hyperbolic polynomials. 

\subsection{Original N-path}
Following Nuij \cite{Nu}, we consider the following operators on
polynomials $\sF_\R \subset \R[t,x]=\R[t,x_1,\ldots,x_n]$. 
\begin{align*}
  T^\ell_s &\colon p\mapsto p+s \ell\frac{\partial p}{\partial
    t}\quad (\ell\in\R[x] \text{ a linear form})\\
  G_s &\colon p\mapsto p(t,sx)\\
  F_s &= (T^{x_1}_s)^d\cdots (T^{x_n}_s)^d\\
  N_s &= F_{1-s}G_s\,,
\end{align*}
where $s\in\R$ is a parameter. For fixed $s$, all of these
are linear operators on $\R[t,x]$ taking the affine-linear subspace
$\sF_\R$ to itself. Clearly, $G_s$ preserves hyperbolicity for any $s\in\R$,
and $G_0(p)=t^d$ for all $p\in\sF_\R$. The operator $F_s$ is used to
``smoothen'' the polynomials along the path $s\mapsto G_s(p)$. The
exact statement is the following.

\begin{Prop}[Nuij \cite{Nu}] \label{Prop:Nuij}
For $s\ge 0$, the operators $T^\ell_s$
  preserve hyperbolicity. Moreover, for $p\in\sH$, we have $N_s(p)\in\interior(\sH)$ for all
  $s\in [0,1)$, with $N_0(p)\in\interior(\sH)$ not depending on $p$
  and $N_1(p)=p$.\qed
\end{Prop}

For $p\in\sF_\R$, we call $[0,1]\ni s\mapsto N_s(p)$ the \emph{N-path} of $p$. 

\begin{Remark}
The N-path defines the contraction that appeared in
Proposition~\ref{Prop:TopologyH}(2): for all $p\in\sH$, the N-path leads to
$N_0(p)=F_1G_0(p)=F_1(t^d)=N_0(t^d)$.

However, it does not define a {\em strong} deformation retract as
$N_s(t^d)=t^d$ holds only at the end of the N-path.
\end{Remark}

In order
to ensure smoothness of the map $\Phi\colon\sM\to\sH$ along an N-path,
we would like to ensure that the N-path stays inside the set
$\sH^\circ$ of smooth hyperbolic polynomials and thus away from the
ramification locus, by
the discussion in the preceding section. 

\begin{Example}
If $n=2$ and $d\le 3$, we know that $\sH^\circ=\interior(\sH)$ (Remark
\ref{Remark:StrictlyHyperbolicSingular}). For $n=d=2$, we verify by explicit computation that the N-path stays
inside the strictly hyperbolic conics, which in this case is just
equivalent to irreducibility. 
  
Let $D=\Diag(d_1,d_2)$ and $R = \begin{bmatrix}
  r_{(1,1)}&r_{(1,2)}\\
  r_{(1,2)}&r_{(2,2)}
\end{bmatrix}.$
The quadric  $$N_s(\Phi(D,R)) = Ax^2+Bxy+Cy^2+Dxt+Eyt+Ft^2$$ is not
contained in $\sH^\circ$ if and only if it factors, which happens if and only if
\begin{eqnarray*}
4\Disc({N_s(\Phi(D,R))}) &=& 4\det
\begin{bmatrix}
  A&B/2&D/2\\
  B/2&C&E/2\\
  D/2&E/2&F
\end{bmatrix} \\
&=&2s^2(s-1)^2(d_1 - d_2)^2+\\
&&2s^2(s-1)^2\left(r_{11}-r_{22}\right)^2 + \\
&&s^4r_{12}^2\left({d_1} - {d_2} \right)^2+\\
&&8r_{12}^2 s^2\left(1-s\right)^2+\\
&&16(s-1)^4\\
&=& 0.
\end{eqnarray*}
The sum-of-squares representation was produced using the intuition obtained by a numerical sum-of-squares decomposition delivered by \textsc{Yalmip} and further exact symbolic computations in \textsc{Mathematica} (see {\tt Nuij-d2.mathematica}~\cite{Leykin-Plaumann:wwwVCN}).
The components of the sum-of-squares decomposition above vanish simultaneously only when $s = 1$ (this proves the proposition) and either $r_{12} = 0$ or $d_2 = d_1$.
\end{Example}

Unfortunately, for $d\ge 4$ it is no longer true that the N-path
always stays inside $\sH^\circ$, due to the existence of strictly hyperbolic
polynomials with complex singularities.

\begin{Example}
  Consider again the polynomial 
\[
p=1/19(19t^4 -31 x^2t^2 - 86 y^2t^2 + 9x^4 + 41 x^2 y^2 +39
  y^4),
\]
which is contained in $\interior(\sH)$ but not in $\sH^\circ$ (c.f.~Remark
\ref{Remark:StrictlyHyperbolicSingular}). One can verify through
direct computation that the polynomial 
\begin{align*} r
= \frac{1}{124659}&\biggl(124659 t^4-221616 t^3 (x+y)-324 t^2
-    \left(205 x^2-912 x y+1580 y^2\right)\\ &+1440 t \left(98 x^3+41
-      x^2 y+316 x y^2+373 y^3\right)\\ &+40 \left(1099 x^4-1568 x^3
-      y+5540 x^2 y^2-5968 x y^3+5849 y^4\right)\biggr)
\end{align*}
is hyperbolic and smooth, i.e.~$r\in H^\circ$, but $N_{1/9}(r)=p$, so
that the N-path for $r$ is not fully contained in
$\sH^\circ$. 
\end{Example}

However, one can still attempt to avoid the ramification points in the following manner.

\begin{Prop}\label{Prop: avoid L}
  The N-path $N_s(p)$, with parameter $s$ varied along a
  piecewise-linear path $[0,c]\cup[c,1) \subset \C$, does not meet the
  ramification locus of $\Phi$, for almost all $c\in \C$. (More
  precisely, this holds for any $c$ taken in the complement of some
  proper real algebraic subset of $\C \simeq \R^2$.)
\end{Prop}
\begin{proof}
  This immediately follows from the fact that the ramification
  locus is a proper complex subvariety of $\sF$ and therefore has real
  codimension at least $2$.
\end{proof}

\begin{Remark}
  Using a random path as described in the proposition may
  result in a non-real determinantal representation: indeed, a
  non-real path for $s$ is not guaranteed to result in a real point in
  the fiber $\Phi^{-1}(N_0(p))$ when a real point in
  $\Phi^{-1}(N_1(p))$ is taken.  
%Nevertheless, the probabilistic
%  algorithm suggested by the proposition gives a way to construct not
%  only complex, but real determinantal representations.

  One may ask for the probability of obtaining a real determinantal
  representation at the end of the path described in
  Proposition~\ref{Prop: avoid L}.  (For a more precise question, one
  may pick $c$ on a unit circle with a uniform distribution.) While
  this probability is clearly non-zero, deriving an explicit lower
  bound seems to be a very hard problem. A naive intuition suggests
  that the probability can be estimated as a ratio of the number of
  real representations (Remark~\ref{rem: real repr}) to the total
  count of complex representations (Theorem~\ref{Thm:U}). The
  experiments with quartic hyperbolic curves, however, suggest that
  the probability is much higher.
\end{Remark}

On the other hand, the set of polynomials for which the N-path avoids
the ramification locus altogether is dense, as the following proposition shows.

\begin{Prop}
  The set of strictly hyperbolic polynomials $p$ such that
  $N_s(p)\in\sH^\circ$ for all $s\in (0,1]$ is dense in $\sH$. (More
  precisely, its complement is a semialgebraic subset of positive codimension.)
\end{Prop}

\begin{proof}
Consider the map
\[
  N\colon\left\{
    \begin{array}{ccc}
      \sF_\R\times\R &\to &\sF_\R\\
      (p,s) &\mapsto &N_s(p).
    \end{array}\right.
\]
and put $\sR=\interior(\sH)\setminus\sH^\circ\subset\sF_\R$. We want
to show that the projection of the semialgebraic set
$N^{-1}(\sR)\subset\sF_\R\times\R$ onto $\sF_\R$ has codimension at
least $1$.

The linear operator $N_s=(T^{x_1}_{1-s})^d\cdots (T^{x_n}_{1-s})^d
G_s$ is bijective for all $s\neq 0$. For $p = -sx_k (\partial
p/\partial t)$ is only possible for $p=0$, since $sx_k(\partial
p/\partial t)$ has strictly lower degree in $t$ than $p$, so
$T^{x_k}_s$ has trivial kernel. Furthermore, $G_s$ is bijective for
$s\neq 0$, hence so is $N_s$. This implies that all fibres $N^{-1}(p)$
of $N$ have dimension at most $1$ (except if $p=N_0(p)$ is the fixed
endpoint of the N-path). In particular, all fibres of $N$ over $\sR$
are at most $1$-dimensional, and since $\sR$ has codimension $2$ in
$\sF_\R$, this implies that $N^{-1}(\sR)$ also has codimension at
least $2$ in $\sF_\R\times\R$, so that the projection of $N^{-1}(\sR)$
onto $\sF_\R$ has codimension at least $1$, as claimed.
\end{proof}

In principle, this proposition can be used as follows: Suppose
$p\in\sH$ is such that the N-path intersects the ramification
locus. If $p\in\interior(\sH)$, we can apply the algorithm to small random
perturbations of $p$, which will avoid the ramification
locus with probability $1$. If $p$ is a boundary point of $\sH$,
i.e.~if $p$ is not strictly hyperbolic, we can first replace $p$ by
$N_{1-\epsilon}(p)$ for small $\epsilon>0$, which is strictly
hyperbolic, and then perturb further if necessary.

On the other hand, we do not know whether the fixed endpoint $N_0(p)=F_1(t^d)$
is contained in $\sH^\circ$ in all degrees $d$, although the direct computation in our experiments
suggests that this is the case. (See also Thm.~\ref{Thm:SmoothRandomEndpoint} below).

\subsection{Randomized N-path}
The following modification of Nuij's construction has proved itself
useful in computations. Let $e=\max\{d,n\}$ and choose a sequence
$L=(\ell_1,\dots,\ell_e)\in\R[x]^e$ of $e$ linear forms. The {\em
  randomized N-path} $\randN_s^{L}$ given by $L$ is defined by
\begin{align*}
\randF_s^{L} &= T^{\ell_1}_s T^{\ell_2}_s\cdots T^{\ell_e}_s\,,\\
\randN_s^{L} &= \randF_{1-s}^{L}\,G_s\,,
\end{align*}
with $T$, $G$ as before. Thus the randomized N-path only involves
$\max\{d,n\}$ differential operators rather than $dn$. In practice,
replacing the N-path by the randomized N-path has worked very well
(c.f.~Remark \ref{Remark:N-path-comparison}).\\
The product of operators in $F_s^L$ can be expanded
explicitly, namely
\[
\randF_s^{L}(p)=\sum_{k=0}^e
s^k\sigma_k(\ell_1,\dots,\ell_e)\frac{\partial}{\partial t^k}p
\]
for all $p\in\sF_\R$, where $\sigma_k(y_1,\dots,y_e)$ denotes the elementary
symmetric polynomial of degree $k$ in the variables $y_1,\dots,y_e$.

\begin{Conjecture}\label{conjecture:main}
Let $e\geq\{d,n\}$. 
\begin{enumerate} 
\item For a given $p\in\sH$, the set of $L$ such that
  $\randN_s^{L}(p) \in \sH^\circ$ for $s\in [0,1)$ is dense in
  $\R[x]^e$.
\item For a general choice of the linear
forms $L\in\R[x]^e$, the set of polynomials $p$ such
that $\randN_s^{L}(p) \in \sH^\circ$ for all $s\in[0,1]$ is dense in
$\sH$. 
\end{enumerate}
\end{Conjecture}

Note that $e$ is chosen
minimally in the sense that if $e<d$ or $e<n$, the fixed endpoint
$\randN_0^{L}(p)=\randF_1^{L}(t^d)$ is no longer smooth or even strictly
hyperbolic. 

We will show that if $n=2$, then at least that endpoint
lies in $\sH^\circ$ for a generic choice of $L$. The
proof relies on Bertini's theorem in the following form.

\begin{Thm}[Bertini's theorem, extended form; {\cite[Thm.~4.1]{Kl}}] On an arbitrary ambient variety, if a linear system has no fixed components, then the general member has no singular points outside of the base locus of the system and of the singular locus of the ambient variety.
\end{Thm}

\noindent For the proof, see also \cite[Thm.~6.6.2]{Jo}.

\begin{Thm}\label{Thm:SmoothRandomEndpoint}
  Let $n=2$. For all $d\ge 2$, the plane projective curve $\sV_\C(\randF_1^L(t^d))$ is smooth for
  a generic choice of $L\in\C[x,y]_1^d$.
\end{Thm}

\begin{proof}
  We first show the following: Let $q\in\C[t,x,y]$ be homogeneous and
  monic in $t$ with $\sV_\C(q)$ smooth, $q\neq t$. Let
  $\ell\in\C[x,y]_1$ and $k$ a positive integer, then $ T_1^\ell(t^kq)
  = t^kq + \ell(kt^{k-1}q + t^kq') = t^{k-1}\bigl((t+k\ell)q + t\ell
  q'\bigr)$, where $q'=\partial q/\partial t$. Put
\[
r_\ell = (t+k\ell)q + t\ell q'.
\] 
We claim that, for generic $\ell$, the variety $\sV_\C(r_\ell)$ is
smooth. To see this, consider $R= (t+u)q + tvq'\in\C[t,x,y,u,v]$. We
find $\partial R/\partial u=q$ and $\partial R/\partial v=tq'$, hence
the singular locus of the variety $\sV_\C(R)$ in $\P^4$ is contained
in $\sV_\C(q)\cap\sV_\C(tq')$. Consider the linear series on
$\sV_\C(R)$ defined by $u = k\ell$, $v=\ell$, $\ell\in\C[x,y]_1$. It
is basepoint-free (in particular without fixed components), since the
only base point on $\P^4$ is $(1:0:0:0:0)$, and that is not a point on
$\sV_\C(R)$. By Bertini's theorem as stated above, the variety
$\sV_\C(r_\ell)$ has no singular points outside the singular locus of
$\sV_\C(R)$ for generic $\ell\in\C[x,y]_1$. Thus we are left with
showing that, for any point $P\in\sV_\C(q)\cap\sV_\C(tq')$, we have
$(\nabla r_\ell)(P)\neq 0$ for generic $\ell\in\C[x,y]_1$. Since
$\sV_\C(q)$ is smooth by assumption, $q$ and $tq'$ are coprime in
$\C[t,x,y]$, hence the intersection $\sV_\C(q)\cap\sV_\C(tq')$ in
$\P^2$ is
finite. If $P$ is any of these intersection points, suppose first that
$q'(P)=0$, which 
implies either $(\partial q/\partial x)(P)\neq 0$ or $(\partial
q/\partial y)(P)\neq 0$, since $\sV_\C(q)$ is smooth. Suppose
$a=(\partial q/\partial x)(P)\neq 0$ and put $b=(\partial q'/\partial
x)(P)$, then
\begin{align*}
(\partial r_\ell/\partial x)(P) &= \bigl(t(P)+k\ell(P)\bigr)a + t(P)\ell(P)b\\
  & = (ka+bt(P))\ell(P) + at(P)
\end{align*}
If $ka\neq -bt(P)$, there is at most one value $\ell(P)$ for which
$(\partial r_\ell/\partial x)(P)=0$. Otherwise, if $ka=-bt(P)$, then $at(P)\neq 0$, so
$(\partial r_\ell/\partial x)(P)\neq 0$. The case $(\partial
q/\partial x)(P)=0$ and $(\partial q/\partial y)(P)\neq 0$ is
analogues. Finally, if $q'(P)\neq 0$, then we must have $t(P)=0$, hence $(\partial r_\ell/\partial
t)(P) = (k+1)\ell(P)q'(P)$ is non-zero, provided that $\ell(P)\neq 0$.

Thus we have shown that $\sV_\C(r_\ell)$ is smooth for generic
$\ell$. To prove the original claim, let $L\in\C[x,y]_1^d$ and
consider
\[
\randF_1^L(t^d)=T_1^{\ell_1}\cdots T_1^{\ell_d}(t^d).
\]
Applying the above, with $k=d$ and $q=1$, shows that
$T_1^{\ell_d}(t^d)$ is of the form $t^{d-1}q$, and $\sV_\C(q)$ is
smooth for generic $\ell_d$. The claim now follows by induction. 
\end{proof}

\section{Algorithm and implementation}\label{Sec:implementation}

Given a hyperbolic polynomial $p\in \sH$, the N-path $N_s(p)$ connects $p=N_1(p)$ with $p_0=N_0(p)$ which does not depend on $p$. This suggests the following algorithm to compute a determinantal representation $(D_p,R_p)\in \sM_\R$ for $p$:
\begin{enumerate}
\item Pick $(D_q,R_q)\in \sM_\R$ giving a strictly hyperbolic
  polynomial $q=\Phi(D_q,R_q)$. Track the homotopy path $N_s(q)$ from
  $s=1$ with the {\em start} solution $(D_q,R_q)$ to $s=0$ producing
  the {\em target} solution $(D_{p_0},R_{p_0})$.  Then $p_0 =
  \Phi(D_{p_0},R_{p_0})$.
\item Track the homotopy path $N_s(p)$ from $s=0$ with the start
  solution $(D_{p_0},R_{p_0})$ to $s=1$ to obtain $(D_p,R_p)$ such
  that $p = \Phi(D_p,R_p)$.
\end{enumerate}

In principle, the first step only has to be performed once
in each degree $d$. In what follows we describe two ways to set up a
polynomial homotopy continuation for the pullback of an N-path.

\subsection{N-path in the monomial basis}
One way is to take the coefficients of the polynomial
$\Phi(D,R)-N_s(p)\in \C[D,R,s][t,x,y]$ with respect to the monomial
basis of $\sF$. This gives a family of square (\#equations=\#unknowns)
systems of polynomial equations in $\C[D,R]$ parametrized by $s$.

Then this family is passed to a homotopy continuation software package
(we use {\sc NAG4M2}~\cite{Leykin:NAG4M2}). As long as $s\in\C$
follows a path that ensures that $N_s(p)$ stays in $\sH^\circ\subset
\sH$, there are no singularities on the homotopy path except, perhaps,
at the target system (see discussion in Section~\ref{Sec:N-path}).

The bottleneck of this approach is the expansion of the determinant in the expression $\Phi(D,R)$ and evaluation of its $(t,x,y)$-coefficients: it takes $\Theta(d!)$ operations and results in an expression with $\Theta(d!)$ terms. This limits us to $d\leq 5$ in the current implementation of this approach.

\subsection{N-path with respect to a dual basis}\label{subsection:dual-basis}
While it may seem that picking a basis of $\sF$ different from the monomial one does not bring any advantage, it turns out to be crucial for practical computation in case of larger $d$.

We fix a {\em dual basis} in $\sF^*$ consisting of $m=\dim\sF$
evaluations $e_i$ at general points $(t_i,x_i,y_i)\in \C^3$, for
$i=1,\ldots,m$. The current implementation generates the points with
coordinates on the unit circle in $\C$ at random.

Now the family of polynomial systems to consider is
$$h_i=e_i(\Phi(D,R)-N_s(p)) \in \C[D,R,s],\ i=1,\ldots,m.$$ Since
$e_i(\Phi(D,R)) = \det(It_i+Dx_i+Ry_i)$, the evaluation of $h_i$ and
its partial derivatives costs $O((\dim \sF)^3) = O(d^6)$. 
 
Evaluation of the (unexpanded) expression $\Phi(D,R)$ and its partial derivatives
is much faster than expanding it in the monomial basis. The latter
costs $\Theta(d!)$ in the worst case and, in addition, numerical tracking
procedures would still need to evaluate the large expanded expression and
its partial derivatives.

We modified the {\sc NAG4M2} implementation of evaluation circuits,
which can be written as straight-line programs, to include taking a
determinant as an atomic operation.

\section{Example}\label{section:example}
The last improvement in the implementation allows us to compute
examples for larger $d$. With an implementation of the homotopy
tracking in arbitrary precision arithmetic, we see no obstacles to
computing determinantal representations for $d$ in double digits.

To give an example, we choose the sextic
\begin{align*}
p =& -36 x^6 - 157 x^4 y^2 - 20 x^3 y^3 - 109 x^2 y^4 + 246 x y^5 -
 92 y^6 - 12 x^3 y^2 t + 90 x^2 y^3 t\\
&+ 10 x y^4 t + 76 y^5 t +
 49 x^4 t^2 + 156 x^2 y^2 t^2 - 16 x y^3 t^2 + 132 y^4 t^2 +
 12 x y^2 t^3\\
&- 14 y^3 t^3 - 14 x^2 t^4 - 27 y^2 t^4 + t^6.
\end{align*}
The polynomial $p$ is hyperbolic, since $p=\Phi(D,R)$ with
\begin{equation} \label{eq:d6-example}
D=\Diag(-3,-2,-1,1,2,3), \ \
R=
\begin{bmatrix}
 0 & 1 & -1 & 1 & 2 & 1 \\
 1 & 0 & -1 & -2 & 1 & -1 \\
 -1 & -1 & 0 & 1 & 2 & 1 \\
 1 & -2 & 1 & 0 & -1 & 1 \\
 2 & 1 & 2 & -1 & 0 & -2 \\
 1 & -1 & 1 & 1 & -2 & 0
\end{bmatrix}.
\end{equation}

Assuming this pair $(D,R)$ is not known, let us describe the
application of our algorithm to recover a determinantal representation
of $p$; one can reproduce the following results by running lines in
{\tt showcase.m2}~\cite{Leykin-Plaumann:wwwVCN}. First, taking an
arbitrary pair $(D_q,R_q)$ and tracking the N-path $N_s(q)$ from the
strictly hyperbolic polynomial $q = \Phi(D_q,R_q) = N_1(q)$ to the
fixed polynomial $p_0 = N_0(q)$, we get
\begin{align*}
 D_{p_0} &= \Diag(.222847, 1.18893, 2.99274, 5.77514, 9.83747, 15.9829)\\
 R_{p_0} & = \begin{bmatrix}6&
      2.51352&
      1.19571&
      4.04309&
      1.42786&
      {-1.98597}\\
      2.51352&
      6&
      3.08656&
      .468873&
      2.38468&
      1.05948\\
      1.19571&
      3.08656&
      6&
      .785785&
      4.66027&
      2.29433\\
      4.04309&
      .468873&
      .785785&
      6&
      1.6226&
      .933245\\
      1.42786&
      2.38468&
      4.66027&
      1.6226&
      6&
      3.50198\\
      {-1.98597}&
      1.05948&
      2.29433&
      .933245&
      3.50198&
      6\\
      \end{bmatrix}
\end{align*}

\noindent Tracking the N-path $N_s(p)$ from $p_0 = N_0(p) = N_0(q)$ to $p = N_1(p)$, we obtain
\begin{align*}
 D' &= \Diag(-3,-2,-1,1,2,3)\\
R' &= \begin{bmatrix}0&
      .596508&
      {-1.43241}&
      2.00316&
      1.10471&
      {-.725394}\\
      .596508&
      0&
      .739773&
      1.79407&
      .0604427&
      {-1.60948}\\
      {-1.43241}&
      .739773&
      0&
      1.56816&
      1.66137&
      {-.165953}\\
      2.00316&
      1.79407&
      1.56816&
      0&
      .839374&
      2.00885\\
      1.10471&
      .0604427&
      1.66137&
      .839374&
      0&
      1.57679\\
      {-.725394}&
      {-1.60948}&
      {-.165953}&
      2.00885&
      1.57679&
      0\\
      \end{bmatrix}
\end{align*}
which is an alternative determinantal representation of $p$. While we
returned to the same point $p$ in the base of the cover $\Phi$, the
route taken has led us to a different sheet than the sheet of the
fiber point $(D,R)\in \Phi^{-1}(p)$ in (\ref{eq:d6-example}) used to
construct this example.

With the default settings of {\sc NAG4M2}, the homotopy tracking algorithm takes 28 steps on the first path and 15 steps on the second. We were not able to find a determinantal representation for this example trying to solve the system $p=\Phi(D,R)$ directly. This is in line with what is reported in~\cite{PSV}: the largest examples that the general solvers could compute with this na\"ive strategy are in degree $d=5$.

\begin{Remark}\label{Remark:N-path-comparison}
The following is a table of experiments that can be reproduced using the examples posted at~\cite{Leykin-Plaumann:wwwVCN}. 
Note that the paths produced by the original and randomized strategies for the same problem are \emph{different paths} and some random choices are made even in the algorithm that follows the original (not randomized) N-path; see~\S\ref{subsection:dual-basis}.  
\def\fail{\text{fail}}
$$
\begin{array}{|c|c|c|c|c|c|}
\hline
d&m=\dim\sF&\text{randomized?}&\text{precision(bits)}&\#\text{steps}&\text{time(seconds)}\\
\hline
6 &    27          & no      &   53     &     22        &   1161      \\ 
6 &    27          & yes     &   53     &     24        &   1326      \\ 
\hdashline
6 &    27          & no      &   53     &  \fail        &          \\ 
6 &    27          & no      &   100    &    38         &   1870       \\ 
6 &    27          & yes     &   53     &    39         &   2098       \\ 
\hline
7 &    35          & no      &   53     &     \fail      &            \\ 
7 &    35          & no      &   100    &     42         &   9273     \\ 
7 &    35          & yes     &   53     &     37         &   7315     \\ 
\hline
8 &    44          & no      &   53     &     27         &   18173     \\ 
8 &    44          & yes     &   53     &     22         &   12091     \\ 
\hline
9 &    54          & no      &   53     &  \fail         &             \\ 
9 &    54          & no      &   100    &     43         &  60692      \\ 
9 &    54          & yes     &   53     &     38         &  43410      \\ 
\hline
10&    65          & no      &   53     &  \fail         &             \\ 
10&    65          & no      &   100    &  \fail         &             \\ 
10&    65          & yes     &   53     &     36         &  163744
%\footnote{\label{foot1}This was obtained on a slower machine with larger memory; the timing was rescaled appropriately.}  
\\ 
\hline
\end{array}
$$
Our empirical conclusion is that using the randomized N-path instead of the original N-path allows computing determinantal representations for larger examples than the original N-path. 

The main underlying reason that may explain the observed lower average practical complexity, in our opinion, is that $\deg_s \randN_s^{L}(p) = d$ while $\deg_s N_s(p) = 2d$, where $d=\deg p$.      

One other reason for a path with lower degree in $s$ having better properties on average is better average conditioning: one may observe that higher precision was needed to finish with the original strategy for several examples. The same increase in precision did not help in our example for $d=10$. 

That said, condition numbers for the fixed end ($s=0$ in Proposition~\ref{Prop:Nuij}) of the original N-path can be better than the correspoding end of the random N-path: the latter depends on the random choices and can be significantly worse if unlucky. We provide two examples for $d=6$ (separated by a dashed line): if the original strategy does not fail then it is even slightly faster than the randomized one. The former is faster near one end of the path ($s=0$), while the later is faster near the other end ($s=1$).      
\end{Remark}

{\small\linespread{1}

}

\end{document}